\documentclass[11pt, a4paper]{article}

\usepackage{amsmath}
\usepackage{amssymb}
\usepackage{amsthm}
\usepackage{appendix}

\newtheorem{theorem}{Theorem}[section]
\newtheorem{corollary}{Corollary}[section]

\newtheorem{remark}{Remark}[section]

\newtheorem{conjecture}{Conjecture}[section]

\begin{document}
\title{The Gaussian Gabor system at the critical density is a weighted lower semi frame}
\author{Elias Zikkos\\
Department of Mathematics, Khalifa University of Science and Technology\\ Abu Dhabi, United Arab Emirates\\
email address:  elias.zikkos@ku.ac.ae and eliaszikkos@yahoo.com}

\maketitle

\begin{abstract}
We show that the Gaussian Gabor system $G(\varphi, \mathbb{Z}^2)$ at the critical density is a weighted lower semi frame for $L^2 (\mathbb{R})$, thus resolving a question left open in \cite{Balazs2026JFAA}. This complements recent results showing that the system is neither a weighted frame nor admits a reproducing partner. 
In fact, after removing any atom from $G(\varphi, \mathbb{Z}^2)$, suitable weights produce a complete Riesz–Fischer sequence and lower semi frame.
\end{abstract}

Keywords: 
Gaussian Gabor system; critical density; weighted lower semi frame;
Riesz--Fischer sequence; Bessel sequence; biorthogonal families.

\smallskip
2020 MSC: 42C15 (primary); 46C05, 47A05 (secondary).

\section{Introduction and Results}
\setcounter{equation}{0}

Let $G(\varphi, \mathbb{Z}^2)$ be the Gaussian Gabor system $G(\varphi, \mathbb{Z}^2)$ at the critical density, that is
\[
G(\varphi, \mathbb{Z}^2)=\{e^{2\pi int}\cdot e^{-\pi(t-m)^2}:\, (m,n)\in \mathbb{Z}^2\}.
\]
It is well known that $G(\varphi, \mathbb{Z}^2)$ has excess equal to one in $L^2 (\mathbb{R})$, that is,
it is complete and remains complete if any one element is removed but not if two.
It is also a Bessel sequence (upper frame) but not a frame. If any element is removed, then the remaining complete system is not a Schauder basis.
\begin{remark}
These classical facts can be found in Christensen \cite[Corollary 11.4.3 and Theorem 11.6.1]{Christensen}  
and Heil \cite[Example 11.34]{Heil}.
\end{remark}

It was recently proved in \cite[Section 6]{Balazs2026JFAA} that $G(\varphi, \mathbb{Z}^2)$ is $not$ a weighted
frame for $L^2 (\mathbb{R})$, relying on the fact, established in \cite[Corollaries 2.6 and 4.3]{Heil2025JFAA},\emph{} 
that $G(\varphi, \mathbb{Z}^2)$ does not have a reproducing partner.
The question whether it could be a $weighted\,\, lower\,\, semi\,\, frame$
for $L^2 (\mathbb{R})$ was left open. We answer this in the affirmative.
In fact, if we delete any atom from the system, the following is true.

\begin{theorem}\label{weightedlower}
For every fixed pair $(a,b)\in\mathbb{Z}^2$, consider the Gabor system
\[
\{e^{2\pi int}\cdot e^{-\pi(t-m)^2}:\, (m,n)\in \mathbb{Z}^2\setminus\{(a,b)\}\}
\]
which is complete and minimal in $L^2 (\mathbb{R})$. Let $\{\gamma_{m,n,a,b}:\, (m,n)\in \mathbb{Z}^2\setminus\{(a,b)\}\}$ be the unique biorthogonal family to this Gabor system in $L^2(\mathbb{R})$. Then
there exist positive constants $c_{m,n}$, so that the weighted Gabor system
\[
\{c_{m,n}\cdot  e^{2\pi int}\cdot e^{-\pi(t-m)^2}:\, (m,n)\in \mathbb{Z}^2\setminus\{(a,b)\}\}
\]
is a complete Riesz-Fischer sequence in $L^2 (\mathbb{R})$ and a lower semi frame for $L^2(\mathbb{R})$.
In particular this holds if we choose the constants $c_{m,n}$ so that
\[
\sum_{(m,n)\in \mathbb{Z}^2\setminus\{(a,b)\}}\frac{||\gamma_{m,n,a,b}||_{L^2(\mathbb{R})}^2}{c_{m,n}^2}<\infty.
\]
\end{theorem}

\begin{corollary}\label{GaborWLF}
The Gaussian Gabor system $G(\varphi, \mathbb{Z}^2)$ is a weighted lower semi frame for $L^2 (\mathbb{R})$.
\end{corollary}

\begin{proof}
Since $G(\varphi, \mathbb{Z}^2)$ has excess equal to one, deleting any atom yields a complete and minimal system in $L^2 (\mathbb{R})$.
By Theorem \ref{weightedlower}, suitable weights make this subsystem a complete Riesz–Fischer sequence and lower semi frame. Hence the full Gaussian Gabor system is a weighted lower semi frame.
\end{proof}

We point out that Theorem \ref{weightedlower} follows immediately
from the following general result, whose proof is given in Section 3. 

\begin{theorem}\label{exact}
Let $\{f_n\}$ be a complete and minimal family in a separable Hilbert space $\cal{H}$, thus having a unique biorthogonal family $\{r_n\}$.
If positive constants $\{c_n\}$ satisfy $\sum||r_n||^2/c_n^2<\infty$, then the family $\{r_n/c_n\}$ is a Bessel sequence in $\cal H$ 
and the family $\{c_n f_n\}$ is a complete Riesz-Fischer sequence in $\cal H$, hence a lower semi frame for $\cal{H}$.
\end{theorem}

\begin{remark}
Theorem \ref{exact} was also proved in \cite[Theorem 1.3]{Zikkos2025ACHA}
but under the stronger condition that the sequence $\{||r_n||/c_n\}$ belongs to the space $\ell^1$. 
The present formulation requires only that $\{||r_n||/c_n\}$ belongs to the space $\ell^2$.
\end{remark}

\begin{corollary}\label{subexact}
Any family in $\cal H$ that contains a subfamily which is complete and minimal in $H$ is a weighted lower semi frame.
\end{corollary}

\begin{remark}
As proved recently in \cite{BalazsShamsabadi}, complete Riesz-Fischer sequences are equivalent to minimal lower frame sequences.
\end{remark}

Theorem \ref{exact} answers partially the following conjecture.

\begin{conjecture}\cite[Conjecture 3.1]{Balazs2026JFAA}
If a family $\cal{F}$ in a separable Hilbert space $\cal{H}$ is complete, then $\cal{F}$ is a weighted lower semi frame
for $\cal{H}$.
\end{conjecture}

The conjecture is known to hold if $\cal{F}$  has a reproducing partner for $\cal{H}$ (\cite[Proposition 3.3]{Balazs2026JFAA}).
A special case is when $\cal{F}$ contains a Schauder basis for $\cal{H}$ or a frame for $\cal{H}$
(\cite[Corollary 3.5]{Balazs2026JFAA}). Corollary \ref{subexact} adds a further
case: any complete family containing a complete and minimal subfamily.

\begin{remark}
The case when F is complete but every complete subfamily has infinite
excess remains open.
\end{remark}
A natural example is the M\"{u}ntz family $\{t^{\lambda_n}\}$, where $\{\lambda_n\}\subset\mathbb{N}$ and $\sum \lambda_n^{-1}=\infty$.  
By the M\"{u}ntz-Sz\'{a}sz theorem the family is complete in $L^2 (0,1)$ if and only if $\sum \lambda_n^{-1}=\infty$.
Thus, every complete subfamily of $\{t^{\lambda_n}\}$ again has an infinite excess in $L^2(0,1)$. Whether such families are weighted lower semi frames for $L^2 (0,1)$ is unknown.

To this end, we note that since $G(\varphi, \mathbb{Z}^2)$ does not have a reproducing partner,
\cite[Proposition 3.3]{Balazs2026JFAA} cannot be applied to deduce whether $G(\varphi, \mathbb{Z}^2)$
is a weighted lower semi frame for $L^2 (\mathbb{R})$.
Theorem \ref{weightedlower} yields Corollary \ref{GaborWLF}, which settles this question for the Gaussian Gabor system.

Another example of the same principle is the following.
\begin{theorem}\cite[Theorem 1.4]{Zikkos2025ACHA}
For any fixed $\alpha\ge\frac{1}{2}$, the system $\{t^{\alpha}\cdot e^{2\pi i n t}\}_{n\in\mathbb{Z}}$ has a finite excess in $L^2(0,1)$,
it does not have a reproducing partner for $L^2 (0,1)$, yet it is a weighted lower semi frame for $L^2 (0,1)$.
\end{theorem}

\section{Background}
\setcounter{equation}{0}

Let $\mathcal{H}$ be a separable Hilbert space endowed with an inner product $\langle .\,  _,\, .\rangle$ and a
norm $||\, .\,  ||$. Let $F:=\{f_n\}$ be a countable family of vectors in $\mathcal{H}$.

$F$ is $complete$ in $\cal H$ if its closed span in $\mathcal{H}$ is equal to $\mathcal{H}$. $F$
is $minimal$ if each vector $f_n$ lies outside the closed span of the remaining elements of $F$ in $\mathcal{H}$.
\begin{remark}
We say that $F$ is exact if it is both complete and minimal.
\end{remark}

The minimality of $F$ is equivalent to the existence of a $biorthogonal$ family,
that is, there is a sequence $\{g_n\}$ in $\mathcal{H}$, so that $\langle f_n , g_m \rangle = \delta_{n,m}$. If $F$ is exact then 
$\{g_n\}$ is unique but it does not have to be exact.

A family $F=\{f_n\}$ is a $frame$ if there exist two positive numbers, $C$ and $D$, so that
\[
C\cdot ||f||^2\le \sum |\langle f, f_n \rangle |^2\le D\cdot ||f||^2,
\qquad \forall\,\, f\in\cal{H}.
\]
If only the left inequality holds then $F$ is a \emph{lower semi frame} and
if only the right inequality holds, then $F$ is a $Bessel$ sequence.

\begin{remark}
A family $F=\{f_n\}$ in $\mathcal{H}$ is called a $\bf weighted$ lower semi frame if there exist a positive number $C$
and real positive numbers $c_n$, so that
\[
C\cdot ||f||^2\le \sum |\langle f, c_n f_n \rangle |^2,\qquad
\qquad \forall\,\, f\in\cal{H}.
\]
\end{remark}

A result attributed to Bari \cite[Chapter 4, Section 2, Theorem 3]{Young} characterizes Bessel sequences:
$F$ is Bessel if and only if there exists a positive number $C$ so that for any finite scalar sequence
$\{\gamma_n\}$ we have
\begin{equation}\label{bessel}
\left|\left|\sum \gamma_n f_n \right|\right|^2 \le C\sum |\gamma_n|^2.
\end{equation}

A family $F=\{f_n\}$ is called a $Riesz-Fischer$ sequence (see \cite[Chapter 4, Section 2]{Young})
if for every sequence $\{a_n\}$ in the space $\ell^2$ the moment problem $\langle f, f_n\rangle =a_n$
has a solution $f\in\cal{H}$. We also have the following characterization 
(\cite[Chapter 4, Section 2, Theorem 3]{Young}): 
$F$ is a Riesz-Fischer sequence in $\cal H$ if and only if
there exists a positive number $C$ so that for any finite scalar sequence $\{\gamma_n\}$ we have
\[
C\sum |\gamma_n|^2   \le \left|\left|\sum \gamma_n f_n \right|\right|^2.
\]
This inequality implies that a Riesz-Fischer sequence is also a minimal sequence.

Finally, we collect results that connect Bessel sequences, Riesz-Fischer sequences and lower semi frames. 
\begin{theorem}\label{Casazza} \cite[Proposition 2.3, (ii)]{Casazza}\label{Cas1}
The Riesz-Fischer sequences in $\cal{H}$ are precisely the families for which a biorthogonal Bessel sequence exists.
\end{theorem}
Complete Riesz-Fischer sequences are equivalent to minimal lower frame sequences: indeed
this follows from the next two results.
\begin{theorem}\label{CasazzaLower}\cite[Theorem 3.2]{Casazza}
Let $\{f_n\}$ be a complete Riesz-Fischer sequence in a Hilbert space $\cal{H}$.
Then $\{f_n\}$ is a minimal sequence and also a lower semi frame for $\cal{H}$.
\end{theorem}
\begin{theorem}\cite[Corollary 6.20]{BalazsShamsabadi}
Let $\{f_n\}$ be a minimal sequence and also a lower semi frame for a Hilbert space $\cal{H}$. Then
$\{f_n\}$ is a complete Riesz-Fischer sequence in $\cal{H}$.
\end{theorem}

\section{Proof of Theorem $\ref{exact}$}
\setcounter{equation}{0}

Let $\{r_n\}$ denote the unique biorthogonal family to the complete minimal system $\{f_n\}$ in $\cal H$. 
Then for any positive constants $\{c_n\}$, the families $\{r_n/c_n\}$ and $\{c_n f_n\}$ are biorthogonal. 
Completeness of $\{f_n\}$  in $\cal H$ passes to the family $\{c_n f_n\}$.

Now, choose positive numbers $\{c_{n}\}$ so that
\[
\sum \frac{||r_{n}||^2}{c^2_{n}}<\infty.
\]
This condition implies that the family $\{r_n/c_n\}$ is a Bessel sequence in $\cal H$ hence by
Theorem \ref{Casazza} its biorthogonal family $\{c_n f_n\}$ is a Riesz-Fischer sequence in $\cal H$. Being also complete in $\cal H$,
Theorem $\ref{CasazzaLower}$ gives that $\{c_n f_n\}$ is a lower semi frame for $\cal H$.

\section{Open problem}

\begin{conjecture}
Preliminary symbolic/numerical exploration suggests that
when $(a,b)=(0,0)$ then
\[
||\gamma_{m,n,0,0}||_{L^2(\mathbb{R})}^2 \asymp \log (e+\sqrt{m^2+n^2}).
\]
Thus if we choose positive constants $c_{m,n}$ so that
\[
\sum_{(m,n)\in \mathbb{Z}^2\setminus\{(0,0)\}}\frac{ \log (e+\sqrt{m^2+n^2})  }{c_{m,n}^2}<\infty,
\]
then the weighted Gabor system
\[
\{c_{m,n}\cdot  e^{2\pi int}\cdot e^{-\pi(t-m)^2}:\, (m,n)\in \mathbb{Z}^2\setminus\{(0,0)\}\}
\]
is a complete Riesz-Fischer sequence in $L^2 (\mathbb{R})$ and a lower semi frame for $L^2(\mathbb{R})$.

At present, the author has not verified these estimates.
\end{conjecture}

\end{document}